\documentclass[12pt, 14paper,reqno]{amsart}
\usepackage{amsmath,amsfonts,amssymb}
\usepackage{mathrsfs}
\usepackage{longtable}
\usepackage[breaklinks]{hyperref}
\usepackage{graphicx}

\theoremstyle{plain}
\newtheorem{thm}{Theorem}[section]

\newtheorem{prop}{Proposition}[section]

\newtheorem{remark}{Remark}[section]

\newtheorem{thma}{Theorem}

\theoremstyle{proof}

\numberwithin{equation}{section}

\begin{document} 
\title[Exponents of class groups of imaginary quadratic fields]{On the exponents of class groups of some families of imaginary quadratic fields}
\author{Azizul Hoque}
\address{Department of Mathematics, Rangapara College, Rangapara, Sonitpur-784505, Assam, India}
\email{ahoque.ms@gmail.com}


\date{\today}


\begin{abstract}
Let $a\geq 1$ and $n>1$ be odd integers. For a given prime $p$, we prove under certain conditions that the class groups of imaginary quadratic fields $\mathbb{Q}(\sqrt{a^2-4p^n})$ have a subgroup isomorphic to $\mathbb{Z}/n\mathbb{Z}$. We also show that this family of fields has infinitely many members with the property that their class groups have a subgroup isomorphic to $\mathbb{Z}/n\mathbb{Z}$. In addition, we deduce some unconditional results concerning the divisibility of the class numbers of certain imaginary quadratic fields. At the end, we provide some numerical examples to verify our results. 
\end{abstract}
\maketitle{}

\section{Introduction}
One of the fundamental problems in the theory of quadratic fields is the following: 

{\it For a given integer $n>1$, find quadratic fields whose class group has a subgroup isomorphic to $\mathbb{Z}/n\mathbb{Z}$.}

The following is a slightly weaker form of this problem:  
{\it Find quadratic fields whose class number is divisible by a given integer $n>1$.}
These problems are useful for understanding the structure of class groups of quadratic fields. 
Nagell \cite{NA22} (resp. Yamamoto \cite{YA70}) proved the existence of infinitely many imaginary (resp. both real and imaginary) quadratic fields whose class number is divisible by a given integer $n\geq 2$.  
Many authors (see, Ankeny and Chowla \cite{AC55}, Chakraborty et al. \cite{CH18}, Soundararajan \cite{SO00} and Kishi \cite{KI09}) gave some families of imaginary quadratic fields of the form $\mathbb{Q}(\sqrt{x^2-y^n})$ whose class number is divisible by $n$. 

In this paper, we focus on the following $3$-parametric family of imaginary quadratic fields:$$K(x,y,n):=\mathbb{Q}(\sqrt{x^2-4y^n}),$$
where $x\geq 1, y\geq 2$ and $n\geq 2$ are integers. 
Gross and Rohrlich were the first to study this family in \cite{GR78} when $x=1$ and $n>1$ is an odd integer, and proved  Theorem \ref{thmGR}. Later, Cohn slightly refined this result in \cite{CO02} only when $y=2$ (see, Theorem \ref{thmC}). In 2009, Louboutin came back in \cite{LO09} to this family and proved that in case of odd $y\geq 3$ with the property that it has at least one prime factor equal to $3$ modulo $4$, the class number of $K(1,y,n)$ is divisible by $n$. For even integer $n\geq 6$, Ishii \cite{IS11} proved that the class number of $K(1,y,n)$ is divisible by $n$, except $(y, n)=(13, 8)$. After that Ito came back to this family in \cite{IT15} and summed up  all these results. She proved that in case of odd $y$, the class number of $K(1, y, n) $ is divisible by $n$, except for some values of $n$ and $y$. She also considered the imaginary quadratic fields $K(3^t, p, n)$ with $p$ prime and $t, n$ positive integers, and investigated the divisibility of the class numbers of these fields by $n$. One of the aims of this paper is to prove a stronger result than existing results concerning the divisibility of the class numbers of imaginary quadratic fields of the form $K(a,p, n)$. Another aim is to expound some of the existing results (see, Theorems \ref{thmGR} and \ref{thmC}) and to deduce some unconditional results (see, Theorems \ref{thm4.1} and \ref{thm4.2}) along the same line.  More precisely, we prove the following results.

\begin{thm}\label{thm1}
Let $a\geq 1$ and $n\geq 3$ be odd integers, and $p$  a prime such that $\gcd(a,p)=1$ and $a^2<4p^n$.
Suppose that $-d$ is the square-free part of $a^2-4p^n$. For $a\ne 1$, assume that one of the following holds:
\begin{itemize}
\item[(i)] $a\not\equiv \pm b\pmod \ell$ for any divisor $b$ of $a$ other than $a$ and for any prime divisor $\ell$ of $n$, and $d\ne 3$,
\item[(ii)] $2^{\ell-1}a\not\equiv b^\ell\pmod d$ for any divisor $b$ of $a$ other than $a$ and for any prime divisor $\ell$ of $n$. 
\end{itemize}
Then except for $(a,p,n)\in\{(5,2,3), (5,2,9), (11,2,5), (13, 2, 7)\}$, the class group of $\mathbb{Q}(\sqrt{-d})$ has a subgroup isomorphic to $\mathbb{Z}/n\mathbb{Z}$.
\end{thm}

We note that $K(a, p, n)=\mathbb{Q}(\sqrt{-d})$. 
 We compute the class numbers of some fields in this family to verify Theorem \ref{thm1} and put them in Table \ref{T1}. We see in Table \ref{T1} that the assumptions (i) and (ii) in Theorem \ref{thm1} hold very often. In addition, for a given odd integer $a\geq 1$, the condition (ii) holds almost always, which can be proved using the remarkable theorem of Siegel theorem on integral points on affine curves. More precisely, we prove the following result to show the infinitude of the imaginary quadratic fields in this family whose class group has a subgroup isomorphic to $\mathbb{Z}/n\mathbb{Z}$. 
\begin{thm}\label{thm2}
Let $a\geq 1$ and $n>1 $ be odd integers. Then the class group of $K(a,p,n)$ has a subgroup isomorphic to $\mathbb{Z}/n\mathbb{Z}$  for infinitely many primes $p$.
\end{thm}

\section{A result of Bugeaud and Shorey with some preliminaries}

In this section, we deduce and restate of some results which are needed in the proof of Theorem \ref{thm1}. 
We first  restate a result of Bugeaud and Shorey \cite{BS01} concerning the number of solutions of a class of Diophantine equations. We need to fix some notations to state this result.

Let $F_k$ (resp. $L_k$) denote the $k^\text{th}$ term in the Fibonacci (resp. Lucas) sequence defined by $F_0=0,   F_1= 1$,
and $F_{k+2}=F_k+F_{k+1}$ (resp. $L_0=2,  L_1=1$, and $L_{k+2}=L_k+L_{k+1}$), where $k\geq 0$ is an integer. 
For $\lambda\in \{1, \sqrt{2}, 2\}$, we define the subsets $\mathcal{F}, \ \mathcal{G_\lambda},\ \mathcal{H_\lambda}\subset \mathbb{N}\times\mathbb{N}\times\mathbb{N}$ by
\begin{align*}
\mathcal{F}&:=\{(F_{k-2\varepsilon},L_{k+\varepsilon},F_k)\,|\,
k\geq 2,\varepsilon\in\{\pm 1\}\},\\
\mathcal{G_\lambda}&:=\{(1,4p^r-1,p)\,|\,\text{$p$ is an odd prime},r\geq 1\},\\
\mathcal{H_\lambda}&:=\left\{(D_1,D_2,p)\,\left|\,
\begin{aligned}
&\text{$D_1$, $D_2$ and $p$ are mutually coprime positive integers with $p$}\\
&\text{an odd prime and there exist positive integers $r$, $s$ such that}\\
&\text{$D_1s^2+D_2=\lambda^2p^r$ and $3D_1s^2-D_2=\pm\lambda^2$}
\end{aligned}\right.\right\},
\end{align*}
except when $\lambda =2$, the condition `odd' on the prime $p$ should be removed from the above notations. Bugeaud and Shorey proved the following result in \cite{BS01}.
\begin{thma}\label{BST}
For a prime $p$, let $D_1$ and $D_2$ be coprime positive integers such that $\gcd(D_1 D_2, p)=1$. Given $\lambda\in \{1, \sqrt{2}, 2\}$, the number of positive integer solutions $(x, y)$ of the Diophantine equation
 \begin{equation}\label{E1}
 D_1x^2+D_2=\lambda^2p^y
 \end{equation}
is at most one, except for $$
(\lambda,D_1,D_2,p)\in\mathcal{E}:=\left\{\begin{aligned}
&(2,13,3,2),(\sqrt 2,7,11,3),(1,2,1,3),(2,7,1,2),\\
&(\sqrt 2,1,1,5),(\sqrt 2,1,1,13),(2,1,3,7)
\end{aligned}\right\}
$$
and $(D_1, D_2, p)\in
\mathcal{F}\cup \mathcal{G_\lambda}\cup \mathcal{H_\lambda}$.
\end{thma}

\begin{remark}\label{rkbs}
In \cite{BS01}, the authors failed to include $(\lambda, D_1, D_2, p)=(2,7,25,2)$  in the set $\mathcal{E}$. This gives two solutions to \eqref{E1}, namely, $(x,y)=(1,3), (17, 9)$. It comes from computation, and it can be confirmed by \cite{LE93} that these are the only solutions in positive integers corresponding to this quadruple. 
\end{remark}

We now deduce the following proposition on the number of positive integer solutions $(x, y)$ of the Diophantine equation,
\begin{equation}\label{propd}
dx^2+a^2=4p^y,
\end{equation}
where $d\geq 1$ and $a\geq 1$ are some fixed coprime integers, and $p$ is a prime number. Note that 
both $a$ and $d$ are odd.
\begin{prop}\label{propbs}
Let $d$ and $a$ be odd positive integers. Then for any prime $p$, \eqref{propd} has at most one solution $(x, y)$ in positive integers, except $(x,y, d,a,p)\in\{ (1,1,7,1,2),(1,3,7,5,2), \\ 
(3,4,7,1,2), (17,9,7,5,2)\}$. 
\end{prop}
We also need the next two results to complete the proof of this proposition. The first result is due to Ljunggren \cite{LJ43} which gives all the solutions of a Diophatine equation, whereas the second one is due of Cohn \cite{Cohn1} which gives  square terms in the Lucas sequence.
\begin{thma}\label{thmLJ}
The solutions of the Diophantine equation,
$$\frac{z^y-1}{z-1}=x^2,~~x>1,z>1, y>2,$$
are $(x, z, y)\in \{(11,3,5),(20, 7, 4)\}$.
\end{thma}

\begin{thma}\label{thmCO}
The only perfect squares appear in the Lucas sequence are $L_1=1$ and $L_3=4$.
\end{thma}

\begin{proof}[\bf Proof of Proposition \ref{propbs}] 
We note that $(\lambda, D_1, D_2, p)=(2,d,a^2,p)$. Assume that $(2,d,a^2,p)\in\mathcal{E}$. Then  $(d, a^2,p)=(7,1,2)$ and thus \eqref{propd} becomes
\begin{equation}\label{E}
7x^2+1=2^{y+2}.
\end{equation} 
If $y\equiv 1\pmod 3$ with $y\ne 1,4$, then \eqref{E} has no solution by Theorem \ref{thmLJ}. For $y=1,4$, \eqref{E} gives 
$x=1, 3$, respectively. Again if $y\equiv 2\pmod 3$, then \eqref{E} implies $7x^2+1=2\times 8^t$ for some positive integer $t$. Reading this modulo $7$, we arrive at an absurdity. Finally if  $y\equiv 0\pmod 3$, then \eqref{E} implies $7x^2+1=4\times 8^t$ for some positive integer $t$. Reading this modulo $7$, we again arrive at an absurdity. Thus, $(x, y)=(1,1), (3,4)$ are the only solutions of \eqref{E}.  Note that the remaining two exceptions come from Remark \ref{rkbs}. 

We now assume that $(d, a^2,p)\in \mathcal{F}$. Then  there exists an integer $k\geq 2$ such that $F_{k-2\varepsilon}=d, L_{k+\varepsilon}=a^2$ and $F_{k}=p$, where $\varepsilon =\pm 1$. Since $a$ is odd, so that $(k,a,\varepsilon)=(2,1,-1)$ by 
 Theorem \ref{thmCO}. Again utilizing $F_k=p$, we arrive at a contradiction. Therefore $(d, a^2,p)\not\in \mathcal{F}$.

Next suppose that $(d, a^2,p)\in \mathcal{G}_2$. Then $4p^r-1=a^2$ for some positive integer $r$. This is not possible, since $4p^r-1\equiv 3\pmod 4$ and $a^2\equiv 1\pmod 4$. Thus, $(d, a^2,p)\not\in \mathcal{G}_2$.

Finally, let $(d, a^2,p)\in \mathcal{H}_2$. Then there are positive integers $r$ and $s$ such that 
\begin{equation}\label{sr1}
ds^2+a^2=4p^r
\end{equation}
and
\begin{equation}\label{sr2}
3ds^2-a^2=\pm 4.
\end{equation}
We see that `$+$' sign is not possible in \eqref{sr2} by reading it modulo $3$, and thus it becomes
\begin{equation*}\label{rs}
3ds^2-a^2=- 4.
\end{equation*}
This together with \eqref{sr1} give:
$$(a+1)(a-1)=3p^r.$$
As $a$ is odd, so that $p=2$ and $r\geq 2$. Thus 
$$(a+1)(a-1)=3\times 2^r, ~~r\geq 2.$$
This implies $(a+1, a-1)\in\left\{\left(2^{r_1}, 3\times 2^{r_2}\right), \left(3\times2^{r_1},  2^{r_2}\right)\right\}$, where $r=r_1+r_2$.  This further implies $1=2^{r_1-1}-3\times 2^{r_2-1}$ or $1=3\times 2^{r_1-1}-2^{r_2-1}$ which gives $r_1=r_2=1$. Thus $a^2+1=12$, which is not possible. This completes the proof. 
\end{proof}

\section{Proof of Theorems \ref{thm1} and \ref{thm2}}
We begin with the following  proposition which is a crucial ingredient in the proof of Theorem \ref{thm1}. 

\begin{prop}\label{propm}
Let $a,p,n$ and $d$ be as in Theorem \ref{thm1} . Assume that $c$ is the positive integer such that  $a^2-4p^n=-c^2d$. Then for any prime divisor $\ell $ of $n$, the element $\alpha =\dfrac{a+c\sqrt{-d}}{2}$ is not an $\ell$-th power of an element in the ring of integers of $K(a,p,n)=\mathbb{Q}(\sqrt{-d})$.
\end{prop}

\begin{proof}
Assume that $\ell$ is a prime divisor of $n$. Since $n$ is odd, so is $\ell$. Note that $-d\equiv 1\pmod 4$.

Suppose that $\alpha$ is an $\ell$-th power of an element in the ring of integers of $K(a,p,n)$. Then  there are integers $u$ and $v$ with same parity such that 
\begin{equation}\label{al}
\frac{a+c\sqrt{-d}}{2}=\left(\frac{u+v\sqrt{-d}}{2}\right)^{\ell}.
\end{equation}
It can be easily checked that both $u$ and $v$ are odd using the fact that $a$ is odd. 
We now equate the real parts in \eqref{al} to get
\begin{equation}\label{rp}
2^{\ell-1}a=u^{\ell}+\sum_{m=1}^{(\ell-1)/2} \binom{\ell}{2m} u^{\ell-2m}v^{2m}d^m.
\end{equation}
This implies $u\mid 2^{\ell-1}a$. Since $u$ is odd, this further implies $u=\pm b$ for some positive integer $b$ such that $b\mid a$.

We first consider $b\ne a$. Then replacing $u=\pm b$ in \eqref{rp} and reading modulo $\ell$, we get
$2^{\ell-1}a\equiv \pm b^{\ell}\pmod \ell$. This further implies by Fermat's little theorem that $a\equiv \pm b\pmod \ell$ which contradicts to (i).  
Again we put $u=\pm b$ in \eqref{rp}, and then read modulo $d$ to get
$2^{\ell-1}a\equiv b^{\ell}\pmod d$ which contradicts to (ii). 

We now consider $b=a$, that is $u=\pm a$. In this case,  \eqref{al} becomes
$$\frac{a+c\sqrt{-d}}{2}=\left(\frac{\pm a+v\sqrt{-d}}{2}\right)^{\ell}.$$ 
Taking norm on both sides, and then using 
\begin{equation}\label{hy}
dc^2+a^2=4p^n,
\end{equation}
we get
\begin{equation}\label{se}
dv^2+a^2=4p^{n/\ell}.
\end{equation}
Since $\ell$ is a prime divisor of $n$, so that  \eqref{hy} and \eqref{se} together give two distinct solutions, namely $(x,y)=(c,n)$ and $(x,y)=(|v|,n/ \ell)$ of \eqref{propd} in positive integers. This is not possible by Proposition \ref{propbs}. Thus, we complete the proof.
\end{proof}

\begin{proof}[\bf Proof of Theorem \ref{thm1}]
Assume that $c$ is the positive integer satisfying $a^2-p^n=-c^2d$. Since $\gcd(a,p)=1$, so that $p$ splits in $K(a, p, n)=\mathbb{Q}(\sqrt{-d})$. We define, $\alpha: =\dfrac{a+c\sqrt{-d}}{2}$. Note that $\alpha$ and $\bar{\alpha}$ are co-prime, and $N(\alpha)=\alpha \bar{\alpha}=p^n$. Then $(\alpha)= \mathfrak{a}^n$, where $\mathfrak{a}$ is a prime ideal in the ring of integers of $K(a, p, n)$ above $p$. Assume that $[\mathfrak{a}]$ is the ideal class containing $\mathfrak{a}$ in the class group of $K(a, p, n)$. Let $m$ be the order of  $[\mathfrak{a}]$. Then $m$ divides $n$ and $[\mathfrak{a}]^m$ is principal.  Thus we can write $n=tm$ for some positive integer $t$ and $[\mathfrak{a}]^m=(\beta)$ for some integer in $K(a, p, n)$. Therefore,
\begin{equation}\label{beta}
(\alpha)=(\beta)^t=(\beta^t). 
\end{equation}
Since $-d\equiv 1\pmod 4$, so that $-d\ne -1$. For (ii), if $d=3$ then the condition $2^{\ell-1}a \not \equiv \pm 1 \pmod{d}$ shows that $3\mid a$ and thus by \eqref{hy}, $p=3$ which contradicts to $\gcd(a,p)=1$. Therefore $d\ne 3$, and hence $K(a, p, n)\not\in\{\mathbb{Q}(\sqrt{-1}), \mathbb{Q}(\sqrt{-3})\}$.  Thus the only units in the ring of integers of
$K(a, p, n)$ are $\pm1$.  Hence, \eqref{beta} gives $\alpha =\pm \beta^{t}$.  Since $n$ is odd, so is $t$ and thus $\alpha= \gamma^t$ for some integer $\gamma$ in $K(a, p, n)$.  By Proposition \ref{propm}, we get $t=1$ and hence $m=n$. This completes the proof. 
\end{proof}

We now give a proof of Theorem \ref{thm2} using the celebrated Siegel's theorem (see, \cite{ES, LS}) on integral points on a affine curve.

\begin{proof}[\bf Proof of Theorem $\ref{thm2}$]
Assume that $a\geq 1$ and $n> 1$ are odd integers. For each prime $p$ coprime to $a$, by Theorem \ref{thm1}, the class group of $K(a,p,n)$ has a subgroup isomorphic to $\mathbb{Z}/n\mathbb{Z}$ provided $2^{\ell-1}a\not\equiv \pm b^\ell \pmod {d}$, where $b(\ne a)$ is a divisor of $a$ and $\ell$ is a prime divisor of $n$. On the other hand, one gets $d\leq 2^{\ell-1}a-b^\ell$ when $2^{\ell-1}a\equiv \pm b^\ell \pmod{d}$. 

Let $D>1$ be an integer. Then the curve, 
\begin{equation}\label{sc}
Dx^2+a^2=4y^n
\end{equation}
is an irreducible algebraic curve (see \cite[Theorem 1B]{WS}) of genus bigger than $0$. Thus we can conclude using Siegel's theorem (see \cite{LS}) that there are only finitely many integral points $(x,y)$ on the curve \eqref{sc}. Therefore, for each $d>1$ there are only finitely many primes $p$ satisfying 
$$dc^2+a^2=4p^n.$$
As $K(a,p,n)=\mathbb{Q}(\sqrt{-d})$, so that for each odd integer $a\geq 1$, there are infinitely many fields $K(a,p,n)$. Furthermore, since $a^2-4p^n=-c^2d$, so that one gets $d>2^{\ell-1}a-b^\ell$ for sufficiently large $p$. Therefore by Theorem \ref{thm1}, we complete the proof.
\end{proof}

\section{Some applications and numerical examples}
In this section, we first deduce some results concerning the divisibility of the class number of some families of imaginary quadratic fields from Theorem \ref{thm1}. Secondly, we discuss some known results and show how these results can be deduced from Theorem \ref{thm1}.  Thirdly, we give some numerical examples supporting our results. Note that by $h(d)$, we mean the class number of $\mathbb{Q}(\sqrt{d})$.

\begin{thm}\label{thm4.1}
Let $p$ and $q$ be distinct primes with $q\geq 3$. For any integer $m\geq 1$, the class number of $\mathbb{Q}(\sqrt{q^2-4p^{q^m}})$ is divisible by $q^m$. 
\end{thm}
We take $a=q$ and $n=q^m$ in Theorem \ref{thm1}.  Then $q^2-4p^{q^m}=-c^2d$. Using Theorem \ref{BST}, we see that $d\ne 3$. Since $q$ is prime, so that one needs to check the condition (i) in Theorem \ref{thm1} only for $b=\pm 1$ and $\ell=q$, that is $q\not\equiv \pm 1\pmod q$. Hence by Theorem \ref{thm1}, $q^m$ divides $ h(q^2-4p^{q^n})$. 

\begin{thm}\label{thm4.2}
Let $m$ be an odd prime and $t\geq 1$ an integer. For a prime $p$, let $q$ be a prime in the set $\{ms+r: s\in \mathbb{N}\cup\{0\}, 2\leq r\leq m-2\}$ such that $q^2-4p^{m^t}=-c^2mD$, where $c$ is an integer and $D$ is a square-free positive integer coprime to $m$. Then $m^t$ divides $h(-mD)$. 
\end{thm}

We first note that this result also holds for $D=1$. Put $a=q$ and $n=m^t$ in Theorem \ref{thm1}. Then $a^2-4p^n=-c^2d$, where $d=mD$. To apply Theorem \ref{thm1}, we need to check the condition (ii) only for $b=\pm1$ and $\ell=m$, since both $q$ and $m$ are primes. 

As $q\in \{ms+r: s\in \mathbb{N}\cup\{0\}, 2\leq r\leq m-2\}$, so that $q\not\equiv\pm 1\pmod m$. Also $2^{m-1}\equiv 1\pmod m$ since $m$ is odd. These together imply $2^{m-1}q\not\equiv \pm 1\pmod m$, which further implies $2^{m-1}q\not\equiv\pm 1\pmod d$. Thus by Theorem \ref{thm1}, $m^t$ divides $h(-mD)$.     

We now look at the following result of Gross and Rohrlich.
\begin{thm}[{\cite[Theorem 5.3]{GR78}}] \label{thmGR}
Let $n>1$ be an odd integer and $p>1$ an integer. Then except for $p=2$, the class numbers of the imaginary quadratic fields $\mathbb{Q}(\sqrt{1-4p^n})$ are divisible by $n$. For $p=2$, the same holds if $n$ is square-free and $\gcd(n, 3)=1$. 
\end{thm}
Later, Cohn proved a general result in \cite{CO02} when $p=2$. Precisely, he proved: 
\begin{thm}[{\cite[Theorem]{CO02}}]\label{thmC}
Let $n \geq 2$ be an integer and $d$ the square-free part of $2^{n+2} -1$. Then except for the case $n = 4$, $n$  divides $h(-d)$.
\end{thm}
We get Theorem \ref{thmGR} (resp. Theorem \ref{thmC}) for prime $p$ (resp. odd $n$) from Theorem \ref{thm1} by putting $a=1$.


We give some numerical examples in Table \ref{T1} to illustrate Theorem \ref{thm1}. We use MAGMA  to compute numerical the class number of $K(a, p, n)$ for $a,p, n\leq 101$. However, we list them in Table \ref{T1} for $a,p\leq 15$ and $n\leq 9$, which means that discriminants do not exceed $13^9$. It is noted that this list does not exhaust all the imaginary quadratic fields $K(a,p,n)$ of discriminant not exceeding $13^9$.  In table \ref{T1}, we use $*$ in the column for class number, $h(d) $ to indicate the failure of condition (ii) of Theorem \ref{thm1} or exceptional case, whereas  $**$ indicates that conditions (i) and (ii) fail to hold. It is also clear from these numerical examples that none of the conditions (i) and (ii) of Theorem \ref{thm1} are necessary in our results. 
\begin{center}
\small
\begin{longtable}{cccccc|cccccc}
\caption{Numerical examples of Theorem \ref{thm1}.} \label{T1} \\

\hline \multicolumn{1}{c}{$n$} & \multicolumn{1}{c}{$a$} & \multicolumn{1}{c}{$p$}& \multicolumn{1}{c}{$a^2-4p^n$} & \multicolumn{1}{c}{$d$}& \multicolumn{1}{c|}{$h(d)$} & \multicolumn{1}{c}{$n$} & \multicolumn{1}{c}{$a$}& \multicolumn{1}{c}{$p$} & \multicolumn{1}{c}{$a^2-4p^n$}&\multicolumn{1}{c}{$d$} & \multicolumn{1}{c}{$h(d)$}\\ \hline 
\endfirsthead

\multicolumn{12}{c}%
{{\bfseries \tablename\ \thetable{} -- continued from previous page}} \\
\hline \multicolumn{1}{c}{$n$} & \multicolumn{1}{c}{$a$} & \multicolumn{1}{c}{$p$}& \multicolumn{1}{c}{$a^2-4p^n$} & \multicolumn{1}{c}{$d$}& \multicolumn{1}{c|}{$h(-d)$} & \multicolumn{1}{c}{$n$} & \multicolumn{1}{c}{$a$}& \multicolumn{1}{c}{$p$} & \multicolumn{1}{c}{$a^2-4p^n$}&\multicolumn{1}{c}{$d$} & \multicolumn{1}{c}{$h(d)$}\\ \hline 
\endhead
\hline \multicolumn{12}{r}{{Continued on next page}} \\ \hline
\endfoot
\hline 
\endlastfoot

3 & 1 & 2 & -31& 31&  3& 		3 & 1 & 3 & -107 & 107& 3 \\
3 & 1 & 5 & -499 & 499& 3& 3 & 1 & 7 & -1371 & 1371& 12\\
3 & 1 & 11 & -5323 & 5323& 15& 3 & 1 & 13&  -8787&  8787 & 12 \\
3 & 3 & 2 & -23 & 23& 3 & 3 & 3 & 5 & -491& 491&  9 \\
3 & 3&  7 & -1363 & 1363& 6 & 3 & 3 & 11 & -5315 & 5315& 18\\
3 & 3 & 13 & -8779 & 8779& 15 & 3 & 5 & 2 & -7 & 7& 1*\\
3& 5 &3 &-83 &83&3& 3 &5 &7 &-1347&1347& 6\\
3 & 5 & 11 & -5299 & 5299& 12 & 3 & 5 & 13 & -8763 & 8763& 24\\
3 &7 &3& -59& 59&3& 3 &7& 5& -451&451& 6\\
3 &7& 11& -5275&211& 3& 3 &7 &13& -8739&971& 15\\
3 &9& 5& -419&419& 9& 3 &9& 7& -1291&1291& 9\\
3 &9& 11& -5243& 107&3& 3 & 9 & 13&  -8707& 8707& 15\\
3 &11 & 5 &-379 & 379& 3& 3 &11 & 7 & -1251 & 139& 3\\
3 & 11 & 13 &-8667 & 107&3 &3 & 13 & 5 & -331 & 331& 3\\
3 & 13 & 7 & -1203 & 1203& 6 & 3 & 13 & 11 & -5155 & 5155& 12 \\
3 & 15 & 7 & -1147& 1147&  6&3 & 15 & 11 & -5099 & 5099& 39\\
3 &15 &13 &-8563& 8563&  9 & 5 &1 &2 &-127&127& 5\\
5 &1 &3 &-971&971& 15& 5 &1 &5& -12499&12499& 20\\
5 &1& 7& -67227&67227& 70&5 &1& 11& -644203&13147& 15\\
5 &1& 13& -1485171&571& 5& 5 &3& 2& -119&119& 10\\
5 &3& 5& -12491&12491& 55&5 &3& 7& -67219&67219& 65 \\
5 &3& 11& -644195&644195& 320&5 &3 &13& -1485163&1485163& 150\\
5 &5 &2 &-103&103& 5& 5 &5& 3 &-947&947& 5\\
5 &5& 7& -67203 &7467&20&5 &5 &11& -644179&644179& 190 \\
5 &5& 13& -1485147&1485147& 280&5 &7& 2& -79&79& 5\\
5 &7& 3& -923&923& 10&5 &7 &5& -12451&12451& 25\\
5 &7& 11& -644155&644155& 140&5 &7 &13& -1485123&1485123& 270\\
5 &9 &2 &-47& 47&5&5 &9 &5 &-12419&12419& 50\\
5 &9& 7& -67147&67147 &30&5 &9 &11 &-644123&644123& 175\\
5 &9 &13& -1485091&1485091& 250 & 5&11 & 2 & -7 & 7& 1* \\
5 &11& 3& -851&851& 10&5 &11& 5& -12379&12379& 25\\
5 &11& 7& -67107&67107& 80&5 & 11 & 13&  -1485051& 1485051& 350 \\
5 &13 &3 &-803&803& 10&5 &13& 5& -12331&12331& 20\\
5 &13 &7 &-67059& 7451&35&5 &13 &11& -644035&644035& 100\\
5 &15& 7& -67003&67003& 45&5 &15& 11& -643979&643979& 330\\
5 &15 &13 &-1484947&1484947& 185&7 &1 &2 &-511&511& 14\\
7 &1& 3& -8747&8747& 21&7 &1 &5 &-312499 &312499&126\\
7 &1& 7& -3294171&366019& 161&7 &1 &11& -77948683&77948683& 903\\
7 &1& 13& -250994067&250994067& 2352&7& 3& 2& -503&503& 21 \\
7& 3& 5& -312491&312491& 168&7& 3& 7& -3294163&3294163& 252 \\
7 &3& 11& -77948675& 3117947&504&7& 3& 13& -250994059 &250994059&2898 \\
7& 5& 2& -487&487& 7& 7 &5& 3& -8723&8723& 28 \\
7 &5& 7& -3294147&3294147& 448& 7& 5& 11& -77948659 &77948659&2884 \\
7& 5& 13& -250994043& 27888227&1820&7& 7& 2& -463&463& 7 \\
7& 7& 3& -8699&8699& 35&7& 7& 5 &-312451&312451& 175\\
7& 7& 11& -77948635 &77948635&1330 &7& 7& 13& -250994019&250994019& 5824\\
7& 9& 2& -431&431& 21& 7& 9& 5& -312419&312419& 238 \\
7& 9& 7& -3294091&3294091& 322 &7& 9& 11& -77948603&77948603& 2940 \\
7 &9& 13& -250993987&250993987& 2387&7& 11& 2& -391& 391&14\\
7 &11& 3& -8627&8627& 21&7& 11& 5& -312379&312379& 84\\
7 &11& 7& -3294051&3294051& 490 &7& 11& 13& -250993947&250993947& 2940 \\
7 &13& 2& -343&7& 1*&7 &13& 3& -8579&8579& 42 \\
7 &13& 5& -312331&312331& 119&7 &13& 7& -3294003&3294003& 308\\
7 &13& 11& -77948515&77948515& 1330&7& 15& 2& -287&287& 14 \\
7& 15& 7 &-3293947&3293947& 252&7& 15& 11& -77948459&77948459& 3990 \\
7& 15& 13& -250993843&250993843& 1540 &9 &1 &2 &-2047& 2047 &18 \\
9 &1 &3 &-78731 &78731 &108 &9 &1& 5& -7812499 &812499& 549\\
9 &1 &7 &-161414427 &161414427 &2160 &9 &1& 11& -9431790763& 9431790763& 16416  \\
9 &1 &13 &-42417997491& 42417997491& 51480&9 &3& 2& -2039& 2039& 45 \\
9& 3& 5& -7812491 &7812491& 765 &9 &3 &7 &-161414419 &161414419 &2997\\
9 &3 &11 &-9431790755 &9431790755 &41796 &9 &3 &13 &-42417997483 &42417997483& 24030 \\
9 &5& 2& -2023 &7 &1** &9 &5& 7 &-161414403& 161414403 &2304\\
9 &5 &11& -9431790739& 9431790739& 20664&9 &5& 13& -42417997467& 42417997467& 51390 \\
9 &7 &2 &-1999 &1999 &27&9 &7 &3 &-78683 &78683 &72 \\
9 &7& 5& -7812451& 7812451& 702 & 9 &7& 11& -9431790715& 9431790715& 18180\\ 
9 &7& 13& -42417997443& 4713110827& 10008&9 &9& 2& -1967& 1967& 36\\
9 &9& 5&-7812419& 7812419& 1800&9 &9 &7 &-161414347 &161414347 &2520 \\
9 &9 &11 &-9431790683 &9431790683& 28584&9 &9 &13 &-42417997411 &42417997411& 47556 \\
9 &11 &2 &-1927 &1927 &18 &9 &11 &3 &-78611& 78611 &90 \\
9 &11& 5& -7812379& 7812379& 405&9 &11 &7 &-161414307 &17934923 &1242 \\
9 &11& 13& -42417997371 &4713110819& 28215 & 9 &13 &2 &-1879 &1879 &27\\
9 &13& 3& -78563& 78563& 54&9 &13& 5 &-7812331 &7812331 &540 \\
9 &13& 7& -161414259& 161414259& 3060&9 &13& 11& -9431790595& 9431790595& 22176 \\
9 &15 &2 &-1823 &1823 &45&9 &15 &7 &-161414203 &161414203& 2052\\
9 &15& 11& -9431790539& 9431790539& 38610 &9 &15& 13& -42417997267 &42417997267 &27954 
 \end{longtable}
\end{center}

\section*{Acknowledgements}
The author is grateful to Professor K. Srinivas for his valuable suggestions and for generous financial support through his MATRICS Project. The author is also grateful to Professor K. Chakraborty for his valuable comments to improve the presentation of the paper. This work was supported by SERB MATRICS Project (No. MTR/2017/001006), Govt. of India.

\end{document}